\documentclass[11pt]{article}

\usepackage[pctex32]{graphics}

\usepackage{amssymb,latexsym,amsmath,epsfig,amsthm}
\newtheorem{lemma}{Lemma} %=\begin{lemma}...\end{lemma}
\newtheorem{theorem}{Theorem}

\newtheorem{remark}{Remark}
\newtheorem{example}{Example}

\begin{document}
	
	\begin{center}
		\uppercase{\bf On Moebius maps which are characterised by the configuration of their dual maps}
		\vskip 20pt
		{\bf Fritz Schweiger} \\
		{ FB Mathematik, Universit\"at Salzburg, Hellbrunnerstr. 34, 5020 Salzburg, Austria}\\
		{\tt fritz.schweiger@sbg.ac.at}\\
	\end{center}
	\vskip 30pt

	\vskip 30pt
	\centerline{\bf Abstract}
	Here we consider piecewise fractional linear maps with three branches. The paper presents a study of invariant measures with densities which can be written as infinite series. These series either have infinitely many poles or they sum up to a function with just one pole. 
	
	\section*{0. Introduction}
	Since the time when R\'enyi's influential paper \cite{Re} appeared the search for densities of invariant measures has seen many papers (a remarkable case of an unknown density of an invariant maesure is given in \cite{DKL}). This paper is a continuation of the investigations  in \cite{Sch06} and \cite{Sch18c} (for some additions and amendments see also \cite{Sch18a} and \cite{Sch18b}). In this paper Moebius maps (= piecewise fractional linear maps) with three branches are discussed. There is a given partition $a_0 = 0 < a_1 < a_2 < a_3 =1$ of the unit interval $B = [0,1]$ such that the map $T: [0,1]  \to [0,1]$ is bijective from $]a_{k-1}, a_k[$, $k =1,2,3$  onto $ ]0,1[$. The inverse map $V_k$,  $k =1,2,3$ is called an {\em (inverse) branch} and is given as $$V_k (x)= \frac{c_k + d_k x}{a_k + b_k x}, \, k =1,2,3. $$
	The Jacobians of these maps are denoted as $$\omega_k = \omega_k(x) = \frac{a_kd_k - b_k c_k}{(a_k + b_k x)^2}, \, k=1,2,3 .$$  We will use the same notation for the associated matrices $$V_k =  \left( \begin{array}{cc}a_k &  b_k
		\\ c_k & d_k \end{array} \right), \, k =1,2,3. $$
	Let $\epsilon = +$ stand for an increasing map and $\epsilon = -$ for a decreasing map. We call the map $T$ of type $(\epsilon_1, \epsilon_2, \epsilon_3)$, $\epsilon_k \in \{+, -\}, k =1,2,3$, if $V_k$ is increasing or decreasing. If the parameters satisfy some conditions the map $T$ is ergodic and admits a $\sigma$-finite invariant measure. The set $B(k) = V_k[0,1] = [a_{k-1}, a_k] $  is called a cylinder. Clearly these conventions also apply to maps with two or more branches.\\
	The question in these investigations is to give an explicit shape of the density of the invariant measure. In \cite{Sch06} essentially two possibilities are discussed. A map $T^{*}:  B^{*} \to B^{*}$ is called a {\em dual map}  if $B^{*}$ is an interval and $T^{*}$ is a fractional linear map whose branches are given by the adjoint  matrices  
	 $$V_k^{*}  =  \left( \begin{array}{cc}a_k &  c_k
		\\ b_k & d_k \end{array} \right), \, k =1,2,3. $$
	Then the function
	$$h(x) = \int_{B^{*}} \frac{dy}{(1+xy)^2}$$
	is (up to a multiplicative factor) the density of the invariant measure for $T$. 
	If there is map $\psi(t) = \frac{B +Dt}{A+Bt}$ such that $ \psi \circ T = T^{*} \circ  \psi$  (or equivalently $ \psi \circ V_k  = V_k^{*} \circ  \psi, \, k=1,2,3$) then the dual is called a {\em natural dual}. In the other case it is called an {\em exceptional dual}. If the set ${B^{*}}$ shrinks to a point $\xi$ one can see this case as a natural dual as well as an exceptional one. In this case the density of the invariant measure is
	$h(x)= \frac{1}{(1+\xi x)^2}$.\\
	However, there are other examples where we can write down an explicit form of the density of the invariant measure. The basic tool is the following lemma.
	\begin{lemma} Let $T: [0,1] \to [ 0,1]$ be a piecewise fractional linear map with two branches $V_\alpha$ and $V_\beta$. Let $h$ be the density of  the invariant measure for the map $T$. Then we define $S:  [0,1] \to [0,1]$  as $$ Sx = T^2x, x \in B(\alpha), Sx = Tx , x \in B(\beta).$$ If $g$ is the density of the invariant measure for $S$ then the relation
		$$h(x) = g(x) + g(V_\alpha x)\omega_\alpha(x) $$ holds.
	\end{lemma}
\begin{proof} This lemma is special case of the use of the so-called {\em jump transformation} (see \cite{Sch95} and \cite{Sch18c}). The proof is easy. 
Note that the inverse branches of $S$ are given as $V_\beta x$, $V_{\alpha\alpha}x = V_\alpha (V_\alpha x) $, and $V_{\alpha \beta}x = V_\alpha( V_\beta x)$. Their Jacobians are denoted by $\omega_\beta(x)$,  $\omega_{\alpha \alpha}(x)= \omega_\alpha(V_\alpha x)\omega_\alpha(x) $, and $\omega_{\alpha\beta}(x)= \omega_\alpha(V_\alpha x)\omega_\beta(x) $. 
If
$$h(x)=  g(x) + g(V_\alpha x)\omega_\alpha (x) $$ then
$$h(V_\alpha x)\omega_\alpha(x)=  g(V_\alpha x)\omega_\alpha(x) + g(V_{\alpha\alpha}x)\omega_{\alpha\alpha}(x)$$ $$ h(V_\beta x)\omega_\beta(x)=  g(V_\beta x)\omega_\beta (x) + g(V_{\alpha \beta} x)\omega_{\alpha \beta}(x). $$
Then $$h(V_\alpha x)\omega_\alpha(x) + h(V_{\beta}x)\omega_\beta (x) = $$
$$ g(V_\alpha x)\omega_\alpha(x) + g(V_{\alpha\alpha}x)\omega_{\alpha\alpha}(x)  +g(V_\beta x)\omega_\beta(x) + g(V_{\alpha\beta}x)\omega_{\alpha\beta}(x) $$
 $$ = g(V_\alpha x)\omega_\alpha(x) + g(x) = h(x).$$
\end{proof}
\begin{remark} 
Clearly, a similar lemma is true if one exchanges the branches $V_\alpha$ and $V_\beta$.
\end{remark}
\begin{remark}
A recent application of the jump transformation is its use for the investigation of the so-called {\em odd-odd continued fraction algorithm see (\cite{KLL})}. 
\end{remark}
Since the density $h(x)$ of a piecewise fractional linear map with two branches can be determined explicitly we use the equation
$$g(x) = h(x) - g(V_\alpha x)\omega_\alpha(x)= h(x)- h(V_\alpha x)\omega_\alpha(x) + g(V_{\alpha\alpha}x)\omega_{\alpha \alpha}(x) - ... $$
to obtain the density $g(x)$ for some piecewise fractional linear maps with three branches as a possibly infinite series. We call this map with three branches constructed as just described the {\em $1$-step extension} of the given map with two branches. To simplify notation the map $T$ has the partition $0 < \frac{1}{2} < 1$. \\
We will discuss three cases of maps $S$ with three branches, which we label by $\lambda$, $\mu$, and $\nu$, namely the types $(+,+,+)$,
$(+,-,-)$, and $(-,+,+)$. In case $(+,+,+)$ the map $T[0,1]$  has an increasing branch $V_\lambda x = \frac{\lambda x}{1 + (2\lambda -1)x}$ and an increasing branch $V_\beta x= \frac{1 + (\beta -1)x}{2+(\beta-2)x}$. The $1$-step extension $S$ should be a map with partition $0 < \frac{1}{2} < \frac{2}{3}<1$. Then $V_\beta (\frac{1}{2}) = \frac{2}{3}$, $\beta =1$, and $V_\beta x = \frac{1}{2-x}$. Furthermore we see that $\mu = \lambda$ and  $\nu =1$.\\
In case $(+,-,-)$ the map $T[0,1]$  has an increasing branch $V_\alpha x = \frac{\alpha x}{1+(2\alpha -1)x}$ and a decreasing branch $V_\nu x =\frac{1 +(\nu -1)x}{1 +(2\nu -1)x}$. The $1$-step extension $S$ should be a map with partition $0 < \frac{1}{3} < \frac{1}{2}<1$. Then $V_\alpha x = \frac{1}{1+x}$, $\lambda =1$, and $\mu = \nu$. \\
In the third case $(-,+,+)$ the map $T[0,1]$  has a decreasing branch $V_\alpha x=\frac{1-x}{2 +(\alpha -2)x} $ and an increasing branch $V_\nu x = \frac{1 +(\nu -1)x}{2 +(\nu-2)x}$. Again the $1$-step extension $S$ should be a map with partition $0 < \frac{1}{3} < \frac{1}{2}<1$. Then $V_\alpha x = \frac{1-x}{2-x}$, $\mu =1$,  and $\nu = \lambda$.\\
The main result of this paper is that a map with three branches which satisfies some additional conditions given by the configuration of the dual map of the $1$-step extension in most cases is a $1$-step extension. The results are connected with the results of \cite{Sch18c} but the focus of this paper is quite different. One asks how far the configuration of the dual map $S^{*}$ determines the structure of $S$. The short last section is devoted to a generalization of the notion of $1$-extension. \\
In the sequel we will label the three branches by the letters $\lambda$, $\mu$, and $\nu$. The two branches of the map $T$ will be labelled by $\lambda$ and $\beta$ if the cylinder $B(\lambda)$ is kept as the first cylinder (the map $T$ jumps over the cylinder $B(\beta)$, which explains the name jump transformation). In a similar way they will be labelled by $\alpha$ and $\nu$ if the cylinder $B(\nu)$ is kept.\\
\begin{remark} Note that the Greek letters $\alpha$, $\beta$, $\lambda$, $\mu$, and $\nu$ will be used to mark the branch with the same letters as parameters. If a parameter is given a numerical value then we keep the letter as the label of the branch, e. g. if  $V_{\lambda}x = \frac{\lambda x}{1 + (2 \lambda -1)x}$ is a branch belonging to $B(\lambda) = [0,\frac{1}{2}]$ and we choose $\lambda =1$ then we write $V_{\lambda}x = \frac{x}{1 + x}$. 
\end{remark}
For the theoretical background of this study we refer to \cite{Sch95}. 

\section*{1. The case $(+,+,+)$}
We study the map  $V_{\lambda}x = \frac{\lambda x}{1 + (2 \lambda -1)x}$ and $ V_\beta x = \frac{1}{2-x}$.  Then the density of the invariant measure  is 
$$h(x)  = \frac{1}{(1- \lambda + (2 \lambda -1)x)(1-x)}.$$
The density of the invariant measure for the $1$-step extension $S$ is
$$g(x)= \frac{1}{1-x}\big(\sum_{n=0}^{\infty}\frac{(-1) 
	^n}{(n-1)\lambda +1 - ((n-2)\lambda +1)x}\big).$$
Now we consider maps $S$ on the partition $0 < \frac{1}{2} < \frac{2}{3}< 1$ with the three branches
$$V_{\lambda}x = \frac{\lambda x}{1 + (2\lambda -1)x}, V_{\mu}x = \frac{1 + (2\mu -1)x}{2 + (3\mu -2)x}, V_{\nu}x = \frac{2 +(\nu -2)x}{3 + (\nu -3)x}.$$
The parameters satisfy $0 < \lambda \leq 1$, $0 < \mu$, and $1 \leq \nu$. This map is piecewise linear  if $\lambda = \frac{1}{2}$, $ \mu = \frac {2}{3}$, and $\nu = 3$. The dual map is given as
$$V^{*}_{\lambda}y = 2\lambda -1 + \lambda y, V^{*}_{\mu}y = \frac{3 \mu -2+ (2\mu -1)y}{2 + y}, V^{*}_{\nu}y = \frac{\nu -3  +(\nu -2)y}{3 + 2y}.$$
A picture of the $1$-step extension suggests the following conditions. If $\eta$ is a fixed point for $V^{*}_{\nu}$ and $\xi$ is the fixed point for $V^{*}_{\lambda}$  then we require $V^{*}_{\mu}\eta= V^{*}_{\lambda}\eta$ and $V^{*}_{\mu}\xi = V^{*}_{\lambda \lambda}\eta$. \\
We start with the following lemma.
\begin{lemma} The fixed point $\xi$  of $V_{\lambda}^{*}$ is given as $\xi = \frac{2 \lambda -1}{1- \lambda}$
	(if $\lambda =1$ this means $\xi = \infty $). The fixed points $\eta$ of $V_{\nu}^{*}$ are $\eta = -1$ and $\eta = \frac{\nu-3}{2}$.
\end{lemma} 
\begin{proof} The value for $\xi$ is immediate. The values for $\eta$ follow from the quadratic equation $2 \eta^2 + (5 - \nu)\eta + 3 - \nu =0$.
	\end{proof}
\begin{lemma}
	If $V^{*}_{\mu}\eta= V^{*}_{\lambda}\eta$ and $\eta = -1$ then $\lambda = \mu$.\\
	If  $V^{*}_{\mu}\eta= V^{*}_{\lambda}\eta$ and $\eta = \frac{\nu-3}{2}$ then $4 \mu \nu = \lambda \nu^2 + 2 \lambda \nu + \lambda$.
\end{lemma}
\begin{proof}  
The condition  $V^{*}_{\mu}\eta= V^{*}_{\lambda}\eta$ implies
$$ \frac{3 \mu - 2 + (2\mu -1) \eta}{2 + \eta} = 2 \lambda - 1 + \lambda \eta.$$
If $\eta = -1$ then $\lambda -1 = \mu -1$. Hence we get $\lambda = \mu$.\\
If $\eta = \frac{\nu-3}{2}$ then we obtain
$$\frac{2 \mu \nu - \nu -1}{\nu +1} = \frac{\lambda \nu + \lambda -2}{2}$$
and the equation
$$ 4 \mu \nu = \lambda \nu^2 + 2\lambda \nu + \lambda = \lambda (\nu +1)^2 .$$
	\end{proof}
\begin{lemma}
Let $\theta$ be the greater fixed point of $V^{*}_{\mu}$. Then $\theta = \xi$ if and only of $\lambda^2 \mu + \lambda = \mu.$
\end{lemma}
\begin{proof} The equation $\theta^2 + (3-2 \mu)\theta + 2 -3 \mu =0$ has the solutions 
	$\theta = \frac{2 \mu -3 + \sqrt{4 \mu^2 +1}}{2}$ and 
	$\theta = \frac{2 \mu -3 - \sqrt{4 \mu^2 +1}}{2}$. Since $V^{*}_{\mu}\eta \geq V^{*}_{\mu}(-1) = \mu -1$ implies $\theta \geq \mu -1$ only the greater root is of interest.\\
	The condition  $\theta = \xi$ means 
	$$ \frac{2 \mu -3 + \sqrt{4 \mu^2 +1}}{2} = \frac{2 \lambda -1}{1-\lambda}.$$
	This is equivalent to
	$$2\lambda \mu + \lambda - 2 \mu +1 = (1-\lambda)\sqrt{4 \mu^2 +1}.$$
	Since both sides are non-negative we can take squares and find equivalently the condition
	$$\lambda^2 \mu + \lambda = \mu.$$
	\end{proof}
\begin{lemma} Assume $V^{*}_{\mu}\eta= V^{*}_{\lambda}\eta$  and suppose additionally $V^{*}_{\mu}\xi= V^{*}_{\lambda \lambda}\eta$. If $\lambda = \mu$ then $\nu = 1$. If $\mu > \lambda$ then the given map $T$ has a natural dual and the density of the invariant measure is $g(x)= \frac{1}{(1+ \xi x)^2}$.
	\end{lemma}
\begin{proof} We calculate $V^{*}_{\mu}\xi = \lambda \mu + \mu -1$ and $V^{*}_{\lambda \lambda}\eta = 2 \lambda^2 + \lambda -1 + \lambda^2 \eta$. If $\eta = -1$ then $\lambda \mu + \mu -1 = \lambda^2 + \lambda -1$. Hence we see again $\lambda = \mu$. If $\eta = \frac{\nu -3}{2}$ then we find
	$$\nu = \frac{2\lambda \mu +2\mu -\lambda^2 -2\lambda}{\lambda^2}.$$ If $\lambda = \mu$ then $\nu = 1$ (and $\eta =-1$). \\
	The equation $4 \mu \nu = \lambda \nu^2 + 2\lambda \nu +\lambda$ gives 
	$$\nu = \frac {2 \mu - \lambda + 2 \sqrt{\mu^2 -\lambda \mu}}{\lambda}.$$ Since $\nu \geq 1$ the negative sign of the square root is not allowed. \\
	Now we equate
	$$ \frac{2\lambda \mu +2\mu -\lambda^2 -2\lambda}{\lambda^2} = \frac {2 \mu - \lambda + 2 \sqrt{\mu^2 -\lambda \mu}}{\lambda}$$ and we find
	$$\mu - \lambda = \lambda \sqrt{\mu^2 - \lambda}.$$ Therefore $\mu \geq \lambda$. We eventually deduce 
	$$\mu - \lambda = \lambda^2 \mu.$$
	If we want to find a natural dual and a suitable map $\psi(t) = \frac{B+Dt}{A+Bt}$ the standard method leads to 
	$$\psi(t) = \frac{2 \lambda -1 + (\lambda \mu + \mu - 4\lambda +1)t}{1-\lambda + (2\lambda -1)t}.$$
	The condition $\psi(0) = \psi(1)$ leads to 
	$\frac{2\lambda -1}{1-\lambda}= \frac{\lambda \mu + \mu - 2 \lambda}{\lambda}$ and to 
	$\mu - \lambda = \lambda^2 \mu$. Hence
	$g(x)= \frac{1}{(1+ \xi x)^2}$.
	\end{proof}
\begin{theorem} Let $S$ be a piecewise fractional linear map of type $(+,+,+)$ which satifies the conditions 
	$$ V^{*}_{\mu}\eta= V^{*}_{\lambda}\eta, \,  V^{*}_{\mu}\xi = V^{*}_{\lambda \lambda}\eta.$$ 
If $\xi = \eta$ then $S$ has a natural dual and the density of the invariant measure is given as
	$$g(x)= \frac{1}{(1+ \xi x)^2}.$$
If $\xi \neq \eta$ then $\lambda = \mu$ and $\nu =1$. The map is a $1$-step extension and  the density of the invariant measure is given as
$$g(x)= \frac{1}{1-x}\big(\sum_{n=0}^{\infty}\frac{(-1) 
	^n}{(n-1)\lambda +1 - ((n-2)\lambda +1)x}\big).$$
\end{theorem}
\begin{proof} The theorem follows from the preceding lemmas.  To obtain the density in the second case one uses that the $n$-th iterate of $V_\beta x =  \frac{1}{2-  x}$ is given as $V_{\beta}^n x = \frac{n - (n-1)x}{n+1 - n x}$. 
\end{proof}
\section*{2. The case $(+,-,-)$}

We start with  the map  $V_{\alpha}x = \frac{x}{1 + x}$ and $V_{\nu}x = \frac{1 + (\nu -1)x}{1 + (2\nu -1)x}$. The density of the associated invariant measure is $$h(x)= \frac{1}{x(1+ (\nu-1)x}.$$
Then the density of the invariant measure of the $1$-step extension is given as
	$$g(x)=\frac{1}{x}  \sum_{n=0}^{\infty} \frac{(-1)^n }{1 + (n+\nu -1)x} .$$	 
Now we consider the map $S$ with the partition $0 < \frac{1}{3} < \frac{1}{2}< 1$ and the three branches
$$V_{\lambda}x= \frac{\lambda x}{1 + (3\lambda -1)x}, \, V_{\mu}x = \frac{1 + (\mu -1)x}{2 + (3\mu -2)x}, V_{\nu}x = \frac{1 + (\nu -1)x}{1 + (2\nu -1)x}.$$
The parameters satisfy $0 < \lambda \leq 1$, $0 < \mu$, and $0 < \nu$. The dual map is given as
$$V^{*}_{\lambda}y = 3\lambda -1 + \lambda y, V^{*}_{\mu}y = \frac{3 \mu -2+ (\mu -1)y}{2 + y}, V^{*}_{\nu}y = \frac{2\nu -1  +(\nu -1)y}{1 + y}.$$
\begin{lemma} The fixed point of $V^{*}_{\lambda}$ is $\xi = \frac{3 \lambda -1}{1-\lambda}$ (for $\lambda =1$ this means $\xi = \infty$).
\end{lemma}
\begin{proof} The proof is immediate.
	\end{proof}
We define $\eta = V^{*}_{\mu}\xi$ and we require the following conditions: $V^{*}_{\mu} \xi = V^{*}_{\nu} \xi $ and $ \ V^{*}_{\mu}\eta = V^{*}_{\nu \nu}\eta$.
\begin{lemma} Under these conditions the equality $\mu = \nu$ is equivalent to $\lambda =1$.
\end{lemma}
\begin{proof} 
	From $\eta = V^{*}_{\mu} \xi = V^{*}_{\nu} \xi $ we deduce
	$$\xi = \frac{1-2 \nu + \eta}{\nu -1 - \eta}= \frac{2 - 3\mu + 2 \eta}{\mu-1 - \eta}.$$
If $ \eta  = -2$ then we obtain $\frac{2 \nu +1}{\nu +1} = \frac{3 \mu +2 }{\mu +1}$ or $\mu \nu + 2 \mu +1 = 0$, a contradiction.\\ 
Since $V^{*}_{\nu \nu}y = \frac{2\nu -1 +\nu y}{2 + y}$  the condition $ \ V^{*}_{\mu}\eta = V^{*}_{\nu \nu}\eta$ implies
$$\frac{3 \mu -2 + (\mu -1)\eta}{2 + \eta} = \frac{2 \nu -1 + \nu \eta}{2 + \eta} .$$ From $\eta +2 \neq 0$ we obtain 
$$ 3 \mu -2 + (\mu -1)\eta = 2 \nu -1 + \nu \eta $$ which gives
$$\eta =  \frac{3 \mu -2\nu -1}{-\mu + \nu +1}.$$ We insert this value into 
	$$\xi = \frac{1-2\nu + \eta}{\nu -1 - \eta} = \frac{2 - 3\mu + 2 \eta }{\mu -1 - \eta}$$ and obtain
	$$\xi = \frac{2 \mu \nu - 2 \nu^2 + 2\mu - 3 \nu}{\nu^2 - \mu \nu - 2 \mu + 2 \nu } = \frac{3 \mu^2 - 3 \mu \nu + \mu - 2\nu}{- \mu^2 + \mu \nu - \mu + \nu}= -2 + \frac{- 2\mu + \nu}{\nu^2 - \mu \nu -2 \mu + 2 \nu}.$$
If $ \mu = \nu$ then we get $\xi = \infty$ which means $\lambda = 1$. 
If $\lambda =1$ then $\xi = \infty$ together with $V^{*}_{\mu} \xi = V^{*}_{\nu} \xi$ implies $ \mu = \nu$.
	\end{proof}
\begin{lemma}
If $\xi < \infty$ then $\eta = \xi$.
\end{lemma}
\begin{proof}
From $V^{*}_{\mu} \xi = V^{*}_{\nu} \xi $ follows
$$ (\mu - \nu)\xi^2 + 4(\mu -\nu) \xi + 3 \mu - 4 \nu = 0.$$
Hence
$$ \xi = -2 + \frac{\sqrt{\mu^2 - \mu \nu} }{\mu - \nu} = -2 + \frac{- 2 \mu + \nu}{\nu^2 - \mu \nu -2 \mu + 2 \nu}.$$
Since $\mu > \nu$ and $\nu^2 - \mu \nu  - \mu + 2\nu =(2 + \nu)(\nu - \mu)$ we obtain
$$ \sqrt{\mu^2 - \mu \nu} = \frac {2\mu - \nu}{2+ \nu}.$$ This relation is equivalent to
$$ 4 \mu^2 - 4 \mu \nu + \mu^2 \nu - \mu \nu^2 = \nu.$$
Since the conditions $\eta = V^{*}_{\mu} \xi = V^{*}_{\nu} \xi $ and $  \eta = \xi$ imply that $\xi$ is a fixed point for  $V^{*}_{\nu}$ and for $ V^{*}_{\mu}$ we see that
$$\xi = \frac{\nu -2 + \sqrt{\nu^2 +4 \nu} }{2} = \frac{\mu -3 + \sqrt{\mu^2 + 6\mu +1} }{2}.$$
Hence we deduce
$$\mu - \nu -1 = \sqrt{\nu^2 +4 \nu} - \sqrt{\mu^2 + 6\mu +1},$$
$$ \mu \nu + 4 \mu + \nu = \sqrt{\nu^2 +4 \nu} \sqrt{\mu^2 + 6\mu +1},$$
and eventually 
$$ 4 \mu^2 - 4 \mu \nu + \mu^2 \nu - \mu \nu^2 = \nu.$$
This shows that this relation is equivalent to $\eta = \xi$.
\end{proof}

\begin{theorem} Let $S$ be a piecewise fractional linear map of type $(+,-,-)$ which satifies the conditions 
	$\eta = V^{*}_{\mu}\xi= V^{*}_{\nu}\xi$ and $  V^{*}_{\mu}\eta = V^{*}_{\nu \nu } \eta.$ 
	If $\xi = \eta$ then $S$ has a natural dual and the density of the invariant measure is given as
	$$g(x)= \frac{1}{(1+ \xi x)^2}.$$
	If $\xi \neq \eta$ then $\nu = \mu$ and $\lambda =1$. The map is a $1$-step extension and  the density of the invariant measure is given as
	$$g(x)= \frac{1}{x} \sum_{n=0}^{\infty}\frac{(-1) 
		^n}{1 + (\nu + n -1)x}.$$
	\end{theorem} 
\begin{proof} The assertione follow from the foregoing discussion. \end{proof}

\section*{3. The case $(-,+,+)$}

It tuns out that this case is different from the cases $(+,+,+)$ and $(+,-,-)$. 
	Let $V_\alpha x = \frac{1-x}{2-x}$ and $V_{\nu}x = \frac{1 + (\nu -1)x}{2+(\nu-2)x}$. We calculate the density of the invariant measure as $$h(x) = \frac{1}{(1+(\nu -2)x)(\nu + (1-\nu)x)}.$$ 
	One further calculates the iterate $V_{\alpha}^nx = \frac{F_{n-1} - F_{n-3}x}{F_{n+1} - F_{n-1}x}$ where $F_n$ denotes the $n$-th Fibonacci number with start values $F_{-2} =1, F_{-1}= 0$ , and $F_1 =1$. Surprisingly, the terms but the last term in the partial sum of the infinite series cancel. Furthermore,  $\lim_{n \to \infty} \frac{F_n}{F_{n-2}} = \sigma +1$ where $\sigma$ satisfies $\sigma^2 = \sigma +1$, $\sigma = \frac{1+\sqrt{5}}{2}$. Then we obtain
	$$g(x) = \frac{2 - \nu}{1+(\nu -2)x} -\frac{1}{\sigma + 1 -x}.$$ Since the density is up to a multiplcative constant we can also write 
	$$g(x) = \frac{1}{(1+ (\nu-2)x)(1+ \xi x)},$$ where $\xi = \frac{-3+\sqrt{5}}{2}$.
	Inspection of the dual map
	$$V^{*}_{\lambda}y = \frac{\nu -3 -y}{3 + y}, V^{*}_{\mu}y = \frac{-1}{3 + y},  V^{*}_{\nu}y = \frac{\nu -2 +(\nu-1)y}{2+y}$$ shows that it is an exceptional dual on the interval with endpoints $\nu -2$ and $\xi$. \\
We now consider the map with three branches
$$ V_{\lambda}x = \frac{1-x}{3  +(\lambda -3)x}, \, V_{\mu}x = \frac{1+ (\mu -1)x}{3 + (2\mu -3)x}, \, V_{\nu}x = \frac{1  +(\nu -1)x }{2 + (\nu -2)x}.$$ 
The dual map is given as
	$$V^{*}_{\lambda}y = \frac{\lambda - 3 -y}{3 +y}, V^{*}_{\mu}y = \frac{2 \mu -3 +(\mu -1)y}{3 + y}, V^{*}_{\nu}y = \frac{\nu -2 + (\nu -1)y}{2 + y}.$$
Let $\eta$ be a fixed point of $V^{*}_{\nu}$ and $\xi$ a fixed point of $V^{*}_{\mu}$. Then $\eta = -1$ or $\eta = \nu -2$ and  $\xi = \frac{\mu -4 + \sqrt{\mu^2 + 4}}{2}$.
\begin{lemma} Suppose $V^{*}_{\lambda}\eta =  V^{*}_{\mu}\eta.$ If $\eta = -1$ then $\lambda = \mu.$
	If $ \eta = \nu -2$ our map has an exceptional dual.
	\end{lemma}
\begin{proof}
If $\eta = -1$ then we find $\lambda = \mu$. If $\eta = \nu -2$ then we find $\lambda = \mu \nu.$\\
If we suspect there is an exceptional dual we have to check $ V^{*}_{\lambda}\xi = V^{*}_{\nu} \xi.$ 
This equation leads to
$$ \nu \xi^2 + (4 \nu - \lambda)\xi + 3 \nu - 2 \lambda = 0.$$
If we insert $\lambda = \mu \nu$ we find (since $\nu \neq 0$) the  equation 
$$ \xi^2 + (4- \mu)\xi + 3  -2\mu = 0$$
which is the fixed point condition for $\xi$. 
\end{proof}
The next lemma is similar.
 \begin{lemma} Suppose $V^{*}_{\lambda}\xi =  V^{*}_{\nu}\xi.$ Then we find $\lambda = \mu \nu$. 
 \end{lemma}
\begin{proof}
The condition $V^{*}_{\lambda}\xi =  V^{*}_{\nu}\xi$ leads to
$$ \nu \xi^2 + (4 \nu - \lambda)\xi + 3 \nu - 2 \lambda = 0.$$
Then we see that   $$\xi = \frac{\lambda -4\nu + \sqrt{4\nu^2 + \lambda^2}}{2\nu} = 
\frac{\mu -4 + \sqrt{\mu^2 + 4}}{2}.$$ 
If $\xi$ is not a rational function of the parameters $\lambda, \mu,\nu$, then we find $4 \nu - \lambda = \nu(4 - \mu)$ and hence $\lambda = \mu \nu$.
 If $\xi$ is a rational function of the parameters $\lambda, \mu,\nu$ we calculate
 $$ - \xi^2 = \frac{4 \nu - \lambda}{\nu}\xi  + \frac{3\nu - 2\lambda}{\nu}= (4-\mu)\xi + 3 - 2\mu$$  
 Then we see that
 $$\xi = \frac{2 \lambda - 2 \mu \nu}{-\lambda + \mu \nu} = -2$$ if $ \lambda \neq \mu \nu$. But $\xi = -2$ is not an allowed value.\\ If $\lambda  = \mu$, then we find
 $$\xi =\frac{\mu -4\nu + \sqrt{4\nu^2 + \mu^2}}{2\nu} = 
 \frac{\mu -4 + \sqrt{\mu^2 + 4}}{2}.$$ This equation leads to $$\mu \nu - \mu = \sqrt{4 \nu^2 + \mu^2}  - \sqrt{\mu^2 \nu^2  + 4\nu^2}$$  and eventually to $1 + \nu^2 = 2\nu$, hence $\nu =1$. 
 \end{proof}
 \begin{lemma} If $V^{*}_{\lambda}\xi =  V^{*}_{\nu}\xi$ and $\xi = \eta$, then $\lambda = \mu \nu$ and $\lambda = \nu^2 -1$.
\end{lemma}
If $ \eta = -1$ then $-1 = \frac{\mu -4 + \sqrt{\mu^2 + 4}}{2}$ leads to $\mu =0$, a contradiction. If $\eta = \nu -2$ then from the equation
$$ \nu -2 = \frac{\mu -4 + \sqrt{\mu^2 + 4}}{2}$$ follows $ \nu^2 = 1 + \mu \nu = \lambda +1$.
\begin{example} (a) We take $\lambda =3$, $ \mu = \frac{3}{2}$, and $\nu = 2$. In this (linear) case we find $\xi = \eta =0$ and $h(x) =1$.\\
(b)  If we take $\lambda =8$, $ \mu = \frac{8}{3}$, and $\nu = 3$, then we have $\xi = \eta =1 $ and $h(x) =\frac{1}{(1+x)^2}$.	
\end{example}
\begin{lemma} If $V^{*}_{\lambda}\eta =  V^{*}_{\mu}\eta$ and $V^{*}_{\lambda}\xi =  V^{*}_{\nu}\xi$ 
	Then $\eta = -1$ implies $\xi \neq \eta$, $\lambda = \mu $, and $\nu = 1$.
	\end{lemma}
This follows from Lemma 10 and Lemma 11. 
\begin{theorem}
Suppose $V^{*}_{\lambda}\eta =  V^{*}_{\mu}\eta$ and $V^{*}_{\lambda}\xi =  V^{*}_{\nu}\xi$ . If $\eta = -1$ then $\lambda = \mu$ and $\nu =1$. If $\eta = \nu -2$, then we find $\lambda = \mu \nu$. It is a $1$-step extension if $\mu =1$ (and therefore $\lambda = \nu)$.\\
In both cases the map with three branches has an exceptional dual. The density of its invariant measure is given as
$$g(x)= \frac{1}{(1+(\nu -2)x)(1 + \xi x)}$$
where $\xi$ is the rightmost fixed point of $V^{*}_{\mu}$.
\end{theorem}
\begin{proof} The proof follows directly from the foregoing discussion.
	\end{proof}
 Note that in the case $\eta = \nu -2$  the conditions  $V^{*}_{\lambda}\eta =  V^{*}_{\mu}\eta$ and $V^{*}_{\lambda}\xi =  V^{*}_{\nu}\xi$ are equivalent. The case of a $1$-step extension is not visible from the conditions of the theorem. Since the case $(-,+,+)$ was not considered in the section on exceptional duals in \cite{Sch06} a proof has been included.
\begin{example}
If $V^{*}_{\lambda}\eta =  V^{*}_{\mu}\eta$ and $\eta =-1$ then $\nu \neq 1$ is possible. An example is given by $\lambda = \mu = \nu = 2$. Then we have 
$$V^{*}_{\lambda}y = \frac{ -1 - y}{3+y}, V^{*}_{\mu}y = \frac{1 + y}{3 + y}, V^{*}_{\nu}y = \frac{y}{2 + y}.$$
Then $V^{*}_{\lambda}(-1) =  V^{*}_{\mu}(-1) = 0$. We find $\xi = \sqrt{2} -1$ but $V^{*}_{\lambda}(\sqrt{2} -1) \neq  V^{*}_{\nu}(\sqrt{2} -1)$.
\end{example}

\section*{4. $n$-step extensions}
The construction of the $1$-step extension of a map with 2 branches which uses $V_\beta$ leads to a map with 3 branches This construction can be iterated using $V_\beta^2 = V_{\beta \beta}$ and leads to a map with 5 branches which we call the $2$-step extension. By iterating this process we construct the {\em $n$-step extension} using $V_\beta^n$ which is map with $2^n +1$ branches. We denote the Jacobian of $V_\beta^n$ by $\omega_\beta^n$.\\
From the equality $g(x)= h(x) - g(V_\beta x)\omega_\beta(x)$ we found
$$g(x)= h(x)- h(V_\beta x)\omega_\beta(x)$$ $$ + h(V_\beta ^2x)\omega_\beta^2(x) - h(V_\beta^3 x)\omega_\beta^3(x)+ h(V_\beta^4 x)\omega_\beta^4(x) - h(V_\beta^5x)\omega_\beta^5(x) + ... .$$
Let $g_n(x)$ be the density of the invariant measure of the $n$-step extension we have the relation
$g(x)= g_{n-1}(x) -g_n(V_\beta^n x)\omega_\beta^n(x)$ where $g_0(x)=h(x) $ and $g_1(x)= g(x)$. Then we can express $g_n(x)$ as a series which uses the function $h(x)$ only. We find
$$g_2(x)= h(x)- h(V_\beta x)\omega_\beta(x)$$ $$ + h(V_\beta^4 x)\omega_\beta^4(x) - h(V_\beta^5x)\omega_\beta^5(x)+ h(V_\beta^8x)\omega_\beta^8(x) - h(V_\beta^9x)\omega_\beta^9(x) + ... .$$
and
$$g_3(x)= h(x)- h(V_\beta x)\omega_\beta(x)$$ $$ + h(V_\beta^8x)\omega_\beta^8(x) - h(V_\beta^9x)\omega_\beta^9(x)+ h(V_\beta^{16}x)\omega_\beta^{16}(x) - h(V_\beta^{17}x)\omega_\beta^{17}(x) + ... .$$
The following theorem is immediate.
\begin{theorem} The sequence $(g_n(x))$ is convergent and
	$$g_{\infty} (x)= \lim_{n \to \infty} g_n(x) = h(x) -  h(V_\beta x)\omega_\beta(x)= h(V_\alpha x)\omega_\alpha(x).$$
\end{theorem}
\begin{remark}
	The function $g_{\infty}(x) = h(V_\alpha x)\omega_\alpha (x)$ is the density of the jump transformation $U$ defined by
	$$Ux = T^nx, \, T^j x \in B(\beta), 0 \leq j \leq n-2, \, T^{n-1}x \in B(\alpha) .$$ This result can be found in \cite{Sch95}. However, the direct proof is easy. 
	$$\sum_{n=0}^{\infty}g_{\infty}(V_\beta^nV_\alpha x)\omega_\beta^n(V_\alpha x)\omega_\alpha(x)$$ $$ = \sum_{n=0}^{\infty}h(V_\beta^nV_\alpha x)\omega_\beta^n(V_\alpha x)\omega_\alpha(x) - \sum_{n=1}^{\infty}h(V_\beta^nV_\alpha x)\omega_\beta^n(V_\alpha x)\omega_\alpha(x)$$  $$ = h(V_\alpha x)\omega_\alpha(x) = g_{\infty} (x).$$
	\end{remark}
\begin{remark} In the case (-,+,+) the series for $g(x)$ reduces to $g(x)=g_1(x) = \frac{1}{(1 + (\nu -2)x)(1 + \xi x)}$. However, $g_2(x)$ is an infinite series. If $\nu =1$ then $h(x)=\frac{1}{1-x}$ and $g(x)= g_1(x) = \frac{1}{(1-x)(1 +\xi x)}$: We find
	$$g_2(x)= \frac{1}{1-x} -\frac{1}{2-x}
 + \frac{2}{5-2x} - \frac{5}{13-5x}+ \frac{13}{34- 13x} - \frac{34}{89 -34x} + ... .$$
 \end{remark} 
	
\end{document}